\documentclass[12pt]{amsart}
\usepackage{cite}
\usepackage{amsmath}
\usepackage{amsthm}
\usepackage{amssymb}
\usepackage{amsfonts}
\usepackage{array}
\usepackage{fullpage}
\usepackage{tikz-cd}
\usepackage{enumerate}
\usepackage{multicol}
\usepackage{geometry}
\usetikzlibrary{shapes.geometric, arrows}
\let\emptyset\varnothing

\theoremstyle{plain}
\newtheorem*{theorem*}{Theorem}
\newtheorem*{principle*}{Principle}
\newtheorem{theorem}{Theorem}
\newtheorem{lemma}[theorem]{Lemma}
\newtheorem{proposition}[theorem]{Proposition}
\newtheorem{corollary}[theorem]{Corollary}

\newtheorem{definition}[theorem]{Definition}
\newtheorem*{lemma*}{Lemma}
\newtheorem*{proposition*}{Proposition}
\newtheorem*{corollary*}{Corollary}
\newtheorem*{conjecture*}{Conjecture}
\newtheorem*{definition*}{Definition}
\theoremstyle{remark}
\newtheorem{remark}[theorem]{Remark}

\theoremstyle{definition}
\newtheorem{example}{Example}[section]

\tikzstyle{startstop} = [rectangle, rounded corners, minimum width=3cm, minimum height=1cm, text centered, text width= 4cm, draw=black]
\tikzstyle{io} = [rectangle, minimum width=3cm, minimum height=1cm, text centered, text width=3cm, draw=black]
\tikzstyle{process} = [rectangle, minimum width=3cm, minimum height=1cm, text centered, text width=4cm, draw=black]
\tikzstyle{decision} = [rectangle, minimum width=3cm, minimum height=1cm, text centered, text width=4cm, draw=black]
\tikzstyle{arrow} = [thick,->,>=stealth]

\def\codim{\operatorname{codim}}

\author{David Wen}
\title{Flops and Mordell-Weil group of Elliptic Threefolds with (4,6,12)-singular fibers}
  \address{David Wen, National Center for Theoretical Sciences, No. 1 Sec. 4 Roosevelt Rd., National Taiwan University , Taipei, 106, Taiwan}
  \email{dwen@ncts.ntu.edu.tw}
 
\begin{document}

\newcommand{\bigslant}[2]{{\raisebox{.2em}{$#1$}\left/\raisebox{-.2em}{$#2$}\right.}}

\thanks{}

\begin{abstract}
Let $f: W \rightarrow T$ be an elliptic threefold that is a Weierstrass model, which is locally defined by $y^2 = x^3 + fx + g$ over $T$, with a singular fiber such that $(f,g,4f^3 + 27g^2)$ vanishes of order $(4,6,12)$ over an isolated point over $T$. Such a fiber can be explicitly resolved to fibers with smaller vanishing order and the resulting model, $Y$, contains a rational elliptic surface, $S$, where some sections of $S$ are flopping curves on $Y$. As a consequence of this arithmetic and geometric connection, we are able to describe some constraints between flops of $Y$ and the properties of the Mordell-Weil group of $S$ and $Y$.
\end{abstract}

\maketitle

\section{Introduction}

Elliptic curves play a significant role in algebraic geometry as they are objects at the crossroads of being study algebraically, arithmetically and geometrically. By extension, elliptic fibrations which are surjective morphisms of varieties whose general fiber is an elliptic curve share similar properties. Specifically, associated to an elliptic fibration, $f:X \rightarrow T$, is a genus $1$ curve $X_\eta$, which is the scheme theoretic fiber over the generic point, $\eta \in T$. This genus $1$ curve is a birational invariant of birationally equivalent elliptic fibrations of $f$. This establishes a link between the arithmetic of $X_\eta$ and the birational geometry of $f:X \rightarrow B$.

This relationship has been studied as seen in \cite{OguisoShioda, Kawamata97}, where if $X_\eta$ turns out to be an elliptic curve then the automorphism group $X_\eta$ gives relative birational automorphisms of $f:X \rightarrow B$ and vice versa. Specifically, in \cite{OguisoShioda}, we see that there is a direct relation between the singular fibers of a rational elliptic surface and the Mordell-Weil group due to geometric constraints. It is this philosophy that we will apply to the threefold case by using the singular fibers to extract information of the Mordell-Weil groups involved.

In particular, this paper focuses on isolated $(4,6,12)$-singular fibers of a Weierstrass threefold over a smooth base. A Weierstrass model is locally defined by $y^2 = x^3 + fx + g$ and we say a fiber, $F_t$, over a point $t \in T$, an isolated $(4,6,12)$-singular fiber if in an open neighborhood of $t$ we have that $(f,g, 4f^3 + 27g^2)$ vanishes of order $(4,6,12)$ only at $t \in T$. This fiber can be locally resolve crepantly $Y \rightarrow W \rightarrow B$ so that the fiber of $Y$ over $b$ contains a rational elliptic surface $S \rightarrow \mathbb{P}^1$. This sets up the main results of this paper:

\begin{theorem}[Theorem \ref{1main}]
\label{firstMain}
Let $C$ be a horizontal curve on $S \rightarrow \mathbb{P}^1$ such that $C \cap Sing(Y) = C \cap Sing(S) = \emptyset$ and which is a flopping curve of $Y \rightarrow B$ then $C$ is a rational section of $S \rightarrow \mathbb{P}^1$
\end{theorem}

The above relates the birational geometry with the arithmetic of elliptic threefold with the Mordell-Weil group of the rational elliptic surface, $S$. A connection can be seen in the following theorem in the generic case of $S$, e.g. if $S$ was formed from an isolated $(4,6,12)$-singular at a normal crossing intersection.

\begin{theorem}[Theorem \ref{2main}]
\label{secondMain}
If there is an infinite number of flopping curves on $S$ then $MW(Y/T)$ has positive rank and there exists an infinite subgroup of $MW(S/\mathbb{P}^1)$ that consists of flopping sections on $Y$.
\end{theorem}

Some applications of the above results include constructing interesting examples in birational geometry. The methods can be a means to constructing different minimal models of elliptic threefold by way of flopping curves on the rational elliptic surface with bounded on the number of models based upon geometric and arithmetic properties of $Y$ and $S$. The above results are also relevant to Physics, in particular F-theory, where studying of the geometric and arithmetic properties of elliptic Calabi-Yau $n$-folds, for $n \leq 4$, is necessary.

\begin{remark}
Lara Anderson, Antonella Grassi, James Gray and Jonathan Heckman independently obtain parts of the argument used to prove theorem \ref{firstMain}, in their work studying non-equidimensional fibers of elliptic Calabi-Yau threefolds. More specifically, they show an analog of proposition \ref{start} and the existence of a flopping curve in a more general setting without assumption of a section as is done here in this paper. The paper containing their result is currently in progress.
\end{remark}

The structure of the paper is as follows. Section 2 reviews the background material of the birational geometry of elliptic fibrations with section and Weierstrass models. Section 3 proves theorem \ref{firstMain} by using the techniques stated of the previous section. Section 4 proves theorem \ref{secondMain} and highlights some application in studying and bounding number of relative minimal models using arithmetic information.

This paper uses standard notation and terminology of birational geometry and minimal model program found in \cite{KollarMori} and the references for Weierstrass models can be found in \cite{Nakayama88}. All varieties will be normal and defined over $\mathbb{C}$.

\section{Background Results} 

This section will review over the basic facts and results of elliptic fibrations with sections, Weierstrass models and birational transformations of Weierstrass models.
\begin{definition}
A variety, $X$, is an elliptic fibration if there is a surjective morphism $f: X \rightarrow B$ such that a general fiber of $f$ is an elliptic curve. An elliptic fibration $f: X \rightarrow B$ has a section if there is a morphism $\sigma:B \rightarrow X$ such that $f \circ \sigma = id_B$.
\end{definition}

\begin{definition}
We say two elliptic fibrations $\pi: X \rightarrow B$ and $\psi:Y \rightarrow T$ are  birationally equivalent if there exist birational maps $f: X \dashrightarrow Y$ and $g: B \dashrightarrow T$ such that the following diagram commutes:
\[
\begin{tikzcd}
\displaystyle
X \arrow[d, "\pi"'] \arrow[r, dashed, "f"]& Y \arrow[d, "\psi"]\\
B \arrow[r, dashed, "g"] & T\\
\end{tikzcd}
\]
\end{definition}

The key elliptic fibrations, in this paper, are Weierstrass models, which are higher dimensional analogs of the Weierstrass equations of elliptic curves. Below are the definition and results that are needed for arguments in later section. References can be made to \cite{Nakayama88} for more general results of Weierstrass models.

\begin{definition}[{\cite[Def. 1.1]{Nakayama88}} ]
Let $B$ be a complex variety and $\mathcal{L}$ a line bundle on $B$. A Weierstrass model over $B$ is the divisor in $\mathbb{P}$ defined by $Y^2Z - X^3 - fXZ^2 - gZ^3$ where:
\begin{itemize}
	\item $f$ and $g$ are sections of $\mathcal{L}^{-4}$ and $\mathcal{L}^{-6}$ and $4f^3 + 27g^2 \neq 0$ as a section in $\mathcal{L}^{-12}$
	\item $\mathbb{P} := Proj_B(\mathcal{O}_S \oplus \mathcal{L}^2 \oplus \mathcal{L}^3)$
	\item $X, Y$ and $Z$ are the sections of $\mathcal{O}_{\mathbb{P}}(1) \otimes \mathcal{L}^{-2}$, $\mathcal{O}_{\mathbb{P}}(1) \otimes \mathcal{L}^{-3}$ and $\mathcal{O}_{\mathbb{P}}(1)$ corresponding to the natural injections:
	\begin{align*}
	\mathcal{L}^2 &\rightarrow \mathcal{O}_B \oplus \mathcal{L}^2 \oplus \mathcal{L}^3\\
	\mathcal{L}^3 &\rightarrow \mathcal{O}_B \oplus \mathcal{L}^2 \oplus \mathcal{L}^3\\
	\mathcal{O}_S &\rightarrow \mathcal{O}_B \oplus \mathcal{L}^2 \oplus \mathcal{L}^3\\
	\end{align*}
\end{itemize}
We denote a Weierstrass model as $W := W(\mathcal{L},f,g)$ and, from the above definition, $W \rightarrow B$ is an elliptic fibration with section, $\sigma: B \rightarrow W$. The discriminant locus is the image of the singular fibers of $W$ on the base $B$ which we denote by $\Delta \subset B$.
\end{definition}

\begin{proposition}[{\cite[(1.2)]{Nakayama88}} ]
\label{WeierstrassCanonical}
Let $\pi: W := W(\mathcal{L},f,g) \rightarrow B$ be a Weierstrass model, then:
\begin{itemize}
	\item If $B$ is normal then $W$ is normal
	\item If $B$ is Gorenstein then $W$ is Gorenstein
	\item $\omega_W \cong \pi^*(\omega_B \otimes \mathcal{L}^{-1})$
\end{itemize}
\end{proposition}

\begin{definition}
Let $W := W(\mathcal{L},f,g) \rightarrow S$ be a Weierstrass model over $S$ and $h: T \rightarrow S$ a morphism such that $h(T) \not\subset Supp(\Delta)$, then we can define the pullback of $W$ to $T$ by $W_T := W(h^*\mathcal{L}, \hat{f}, \hat{g}) \rightarrow T$ where $\hat{f}$ and $\hat{g}$ is the pullback of $f$ and $g$ by $h$.
\end{definition}

\begin{definition}[{\cite[Def. 1.6]{Nakayama88}}]
Given a Weierstrass model, $W(\mathcal{L},f,g) \rightarrow B$ and $y \in B$, then we call the fiber over $y$ an $(a,b,d)$-fiber where $a,b$ and $d$ are the respective order of vanishing of $f,g$ and $\Delta$ at $y$. We say that $W(\mathcal{L},f,g)$ is a minimal Weierstrass model if there is no prime divisor $D$ on $B$ such that $div(f) \geq 4D$ and $div(g) \geq 6D$.
\end{definition}

\begin{remark}
The $(n,m,d)$-fibers with $n < 4$ or $m < 6$ will correspond to fibers which appear in Kodaira's classification of singular fibers on relative minimal elliptic surfaces.
\end{remark}

\begin{proposition}[{\cite[Thm. 2.5]{Nakayama88}}]
Any Weierstrass model is  birationally equivalent to a minimal Weierstrass model
\end{proposition}

\begin{remark}
The argument, of the above proposition, modifies the line bundle used to define the Weierstrass model, $W(\mathcal{L},f,g)$, by twisting it by the line bundle associated to $D$ where $div(f) \geq 4D$ and $div(g) \geq 6D$. This produces a new Weierstrass model $\hat{W} := W(\mathcal{L} \otimes \mathcal{O}_B(D), \hat{f}, \hat{g})$ where $\epsilon^4\hat{f} = f$ and $\epsilon^6\hat{g} = g$ for some section $\epsilon \in H^0(\mathcal{O}_B(D))$. 

We know that $\hat{W}$ is birational to $W$, but we actually have something stronger since the above twisting can be realized as a weighted blow up of type $(2,3,1)$. So there is, in fact, a birational morphism $\hat{W} \rightarrow W$ commuting with the map to $B$. Naturally, we should expect the twisting to modify canonical bundle of the Weierstrass model.
\end{remark}

\begin{example}
Consider the following elliptic surface defined by the Weierstrass model:
\[
S := W(\mathcal{O}_{\mathbb{P}^1}(-p), s^4, s^6) \rightarrow \mathbb{P}^1
\]
where $s \in \mathcal{O}_{\mathbb{P}^1}(p)$ is a section that vanishes at some point $p \in \mathbb{P}^1$. Thus the equation defining the surface is:
\[
V(ZY^2 - X^3 + s^4XZ^2 + s^6Z^3) \subset \mathbb{P}
\]
where $\mathbb{P}$ is the $\mathbb{P}^2$-bundle over $\mathbb{P}^1$. Without loss of generalities, we can work locally with $Z = 1$ and $p = 0 \in \mathbb{C}$ giving the following equation in $\mathbb{C}^3$
\[
y^2 = x^3 + s^4x + s^6
\]
and the singularity is at $x = y = s = 0$. To resolve we apply a weighted blow up to $\mathbb{C}^3$ with weights $(2,3,1)$ to $(x,y,s)$. This gives the following equations, in the coordinates $((x,y,s),[u,v,w]) \in \mathbb{C}^3 \times \mathbb{P}(2,3,1)$,
\begin{align*}
wx = s^2u \hspace{1 cm} wy = s^3v \hspace{1 cm} vx^3 = uy^2
\end{align*}
which defines the blow up with weights $(2,3,1)$ at $0 \in \mathbb{C}^3$. One can check that the strict transform of $S$ does not contain the singular points of the exceptional $\mathbb{P}(2,3,1)$ over $0$. But in particular we can see that in the $w$-chart, with $x = s^2u$ and $y = s^3v$, the strict transformation has the following form:
\[
v^2 - u^3 - u - 1
\]
This shows that the strict transform of $S$ obtain from the weighted blow up will, in fact, be isomorphic to the Weierstrass model $\hat{S} := W(\mathcal{O}_{\mathbb{P}^1}, 1, 1)$ by \cite[Thm. 2.5]{Nakayama88}. Thus we have a birational morphism $\hat{S} \rightarrow S$ and comparing the canonical divisor of $S$ and $\hat{S}$, by using proposition \ref{WeierstrassCanonical}, we can see that $S$ has log canonical singularities.
\end{example}

\begin{corollary}
\label{CanFiber}
Let $\pi: W:= W(\mathcal{L},f,g) \rightarrow B$ be a Weierstrass model and $D$ a prime divisor on $B$ such that $div(f) \geq 4D$ and $div(g) \geq 6D$, then there is a Weierstrass model $\hat{\pi}: \hat{W} := W(\hat{\mathcal{L}}, \hat{f}, \hat{g})$ such that:
\begin{itemize}
	\item $\hat{\mathcal{L}} \cong \mathcal{L} \otimes \mathcal{O}_B(D)$
	\item There exists a section $\epsilon \in \mathcal{O}_B(D)$ such that $\epsilon^4\hat{f} = f$ and $\epsilon^6\hat{g} = g$.
	\item There exists a birational morphism $\phi:\hat{W} \rightarrow W$ that commutes with the morphisms over $B$.
	\item $\omega_{\hat{W}} = \phi^*(\omega_{W} \otimes \pi^*(\mathcal{O}_B(-D)))$.
\end{itemize}
\end{corollary}


\begin{definition}
Given a Weierstrass model $\pi: W := W(\mathcal{L},f,g) \rightarrow B$, with $b \in B$ and $W_b$ the fiber over $b$, we say that $W_b$ is an isolated $(n,m,d)$-fiber if there is no prime divisor $D$ on $B$ that contains $b$ and the general fiber over $D$ is a $(n,m,d)$-fiber.
\end{definition}

\begin{example}[{\cite[Example 3.8]{Nakayama88}}]
\label{example}
Let $T \cong \mathbb{P}^2$, $\mathcal{L} = \omega_T$ and $\alpha, \beta \in H^0(S,\omega_T^{-1})$ such that $\alpha$ and $\beta$ are general global sections of $\mathcal{L}^{-1}$. Consider the Weierstrass model $W_T := W(\omega_T, \alpha^4, \beta^6) \rightarrow T$ with isolated $(4,6,12)$ singular fibers over the $9$ points of transverse intersections of $\alpha$ and $\beta$ on $S$.

Let $\hat{T}$ be the blow up of those $9$ points and $W_{\hat{T}}$ be the corresponding Weierstrass model pulling back $W_T \rightarrow T$ by base change. Then $W_{\hat{T}} \rightarrow \hat{T}$ is not a minimal Weierstrass model since over each exceptional curve of $\hat{T} \rightarrow T$ the general fiber is a $(4,6,12)$ singular fiber. Thus we can take the minimal Weierstrass model $\hat{W}_{\hat{T}} \rightarrow W_{\hat{T}} \rightarrow W_T \rightarrow T$ and we will see that $\hat{W}_{\hat{T}}$ is non-singular and $\omega_{\hat{W}_{\hat{T}}}$ is the pulback of $\omega_{W_T}$ in the map above. Thus, $\hat{W}_{\hat{T}}$ is a crepant resolution of $W_T$.
\end{example}

\begin{corollary}
Let $f:W \rightarrow B$ be a Weierstrass model with $D$ a prime divisor on $B$ such that the general fiber over $D$ is a $(n,m,d)$ with $n\geq4$, $m\geq 6$ and $d \geq 12$, then the total discrepancy of $W$ is less than or equal to $-1$. 
\end{corollary}

\begin{remark}
This can be generalized with higher orders of vanishing in high codimensions. For example, we can have $(n,m,24)$ singular fibers supported on $V \subset B$ with $\codim(V) = 2$.
\end{remark}

The above in particular highlights that the vanishing of $(f,g,\Delta)$ has a threshold associate to multiples of $(4,6,12)$ which influences the geometry. Since $\Delta = 4f^3 + 27g^2$ the vanishing of $\Delta$ is at a point is dependent on both $f$ and $g$. This sets up the definition of the singular fibers relevant in this paper.
\begin{definition}
A $(4,6,12)$-fiber is a fiber of a Weierstrass model of type $(n,m,12)$ where $n = 4$ and $m \geq 6$ or $n \geq 4$ and $m = 6$.  
\end{definition}

\section{Flops of Elliptic Threefold with $(4,6,12)$-singular fiber}

In this section, we will prove Theorem \ref{firstMain} by studying a resolution of a Weierstrass model with an isolated $(4,6,12)$-singular fiber. For the rest of the paper, we assume that $\pi: W := W(\mathcal{L},f,g) \rightarrow T$ is a minimal Weierstrass threefold with at worst canonical singularities, with $T$ being smooth and containing a point $0 \in T$ which supports an isolated singular fibers of type $(n,m,12)$. The assumption of canonical singularities on the Weierstrass model will not conflict with the existences of the isolated $(4,6,12)$-fiber since the following will show that the isolated $(4,6,12)$-fibers can be resolved crepantly. Notation and terminology of birational geometry and the minimal model program will follow \cite{KollarMori}. 

\subsection{Resolving an Isolated $(4,6,12)$-fiber}
We start with the following proposition to obtain a nice resolution of $W$.
\begin{proposition}
\label{start}
Let $0 \in T$ such that $W_0$, the fiber over $0$, is an isolated $(4,6,12)$-fiber, then there is a terminal model of $W$ that contains a rational elliptic surface, $S \rightarrow \mathbb{P}^1$ which maps to $0 \in T$.
\end{proposition}

Before starting the proof, consider the following commutative diagram:
\[
\begin{tikzcd}
\displaystyle
W \arrow[d, "\pi"'] & \hat{W} := W\times_T \hat{T} \arrow[d, "\hat{\pi}"] \arrow[l, "\phi"] & X \arrow[d, "f"] \arrow[l, "\psi"] \arrow[bend right=26, swap]{ll}{h} & S \arrow[d, "f|_S"] \arrow[l, hook'] \\
T & \hat{T} \cong Bl_0(T) \arrow[l, "g"] & \hat{T} \arrow[l, "\sim"'] & E \cong \mathbb{P}^1 \arrow[l,hook']\\
\end{tikzcd}
\]
with the following properties:
\begin{itemize}
	\item $0 \in T$ supports an isolated $(4,6,12)$-fiber.
	\item $\hat{T} := Bl_0(T)$ is the blow up of $T$ at $0 \in T$ with exceptional divisor $E \cong \mathbb{P}^1$.
	\item $X$ is the minimal Weierstrass model associated to $\hat{W}$ (which is not minimal by construction) and $\psi$ is the weighted blow up reducing the order of vanishing along the divisor supporting the $(4,6,12)$-fiber
	\item $K_X = h^*(K_W)$
	\item $S \rightarrow \mathbb{P}^1$ is an elliptic surface with irreducible fibers
\end{itemize}

Most of the properties can be immediately verified with the exception of the fact that $h$ is a crepant morphism which we show with the lemma below.

\begin{lemma}
\label{crep}
In the above commutative diagram, $h := \phi \circ \psi: X \rightarrow W$ is crepant.
\end{lemma}

\begin{proof}
It is sufficient to show that $\omega_X = h^*(\omega_W)$. The above diagram and the definition of Weierstrass models gives the following computations:
\begin{align*}
\omega_{\hat{W}} &\cong \hat{\pi}^*(\omega_{\hat{T}} \otimes g^*(\mathcal{L}^{-1}))\\
&\cong \hat{\pi}^*(g^*(\omega_T) \otimes \mathcal{O}_{\hat{T}}(E) \otimes g^*(\mathcal{L}^{-1}))\\
&\cong \hat{\pi}^*(g^*(\omega_T \otimes \mathcal{L}^{-1}) \otimes \mathcal{O}_{\hat{T}}(E))\\
&\cong \hat{\pi}^*(g^*(\omega_T \otimes \mathcal{L}^{-1})) \otimes \hat{\pi}^*(\mathcal{O}_{\hat{T}}(E))
\end{align*}
By corollary \ref{CanFiber}, we get the following $\omega_X \cong \psi(\omega_{\hat{W}} \otimes \hat{\pi}^*(\mathcal{O}_{\hat{T}}(-E)))$, and combining all this results in:
\[
\omega_X \cong \psi^*(\hat{\pi}^*(g^*(\omega_T \otimes \mathcal{L}^{-1}))) \cong \psi^*(\phi^*(\pi^*(\omega_T \otimes \mathcal{L}^{-1}))) \cong \psi^*(\phi^*(\omega_W)) \cong h^*(\omega_W)
\]
Thus $h: X \rightarrow W$ is crepant.
\end{proof}

In the above diagram, contrary to what is stated in proposition \ref{start}, it is completely possible that the elliptic surface $S$ is not rational. Consider the following example:
\begin{example}
\label{example1}
Consider the Weierstrass model defined locally over $(s,t) \in \mathbb{C}^2$ by:
\[
y^2 = x^3 + (s - t^2)^2(s + t^2)^2 x + (s - t^2)^3(s + t^2)^3
\]
then generically on the curves defined by $s = t^2$ and $s = -t^2$, $(f,g,\Delta)$ vanish of order $(2,3,6)$ and thus the general singular fibers is of type $I_0^*$. At the origin, $(0,0)$, the vanishing order is $(4,6,12)$ and so is an isolated $(4,6,12)$-fiber. We are in the exact situation as in the assumption of the theorem. Thus we can apply the above procedure as described in the diagram and blow up the origin and reduce the order of vanishing along the exceptional divisor explicitely. Using the transformation, $s = ut$, the the equations becomes:
\[
y^2 = x^3 + t^4(u - t)^2(u + t)^2 x + t^6(u - t)^3(u + t)^3
\]
where $t = 0$ is the exceptional divisor. Reducing the orders of vanishing by a weighted blow up gives the following strict transform:
\[
y^2 = x^3 + (u - t)^2(u + t)^2 x + (u - t)^3(u + t)^3
\]
So that generically over $u = t$ and $u = -t$ we again have $I_0^*$ fibers and over $(0,0)$ we again have an isolated $(4,6,12)$-fiber. The elliptic surface $S$ defined over the exceptional divisor $t = 0$ will turn out to be birational to a product $E \times \mathbb{P}^1$ where $E$ is an elliptic curve and so $S$ is not a rational surface.
\end{example}

The example above shows that the resulting surface $S$ need not be rational but highlights that this is due to the fact that the singular fibers are collected and defined over a single point (the origin in the above example) thus producing another isolated $(4,6,12)$-fiber. We can thus repeat the procedure above to resolve this new isolated $(4,6,12)$-fiber and this must eventually result in a surface that does not have an isolated $(4,6,12)$-fiber. This is due to the fact the $(4,6,12)$-fiber is formed by intersections of divisors and a finite number of blow ups will eventually separate a divisor from the rest of the divisors forming the isolated $(4,6,12)$-fiber.

We can thus generalize the diagram into the following:
\[
\begin{tikzcd}
\displaystyle
& && Y \arrow[d] & \hat{S} \arrow[l, hook'] \arrow[d]\\
W \arrow[d, "\pi"'] & ...  \arrow[l]& \hat{W} := W\times_T \hat{T} \arrow[d, "\hat{\pi}"] \arrow[l] & X \arrow[d, "f"] \arrow[l] \arrow[bend right=26, swap]{lll}{h} & S \arrow[d, "f_S"] \arrow[l, hook'] \\
T & ... \arrow[l] & \hat{T} \arrow[l]& \hat{T} \arrow[l, "\sim"'] \arrow[bend left =26]{lll}{g} & E \cong \mathbb{P}^1 \arrow[l,hook']\\
\end{tikzcd}
\]
with the following properties:
\begin{itemize}
	\item $0 \in T$ supports an isolated $(4,6,12)$-fiber.
	\item $\hat{T}$ is obtain from a sequence of blow ups of $T$ at $0 \in T$ with an exceptional divisor $E \cong \mathbb{P}^1$ that supports an elliptic surface $S \rightarrow E$ that does not have a $(4,6,12)$-fiber.
	\item $X$ is the minimal Weierstrass model associated to $\hat{W}$ (which is not minimal by construction) and $\psi$ is the weighted blow up reducing the order of vanishing along the divisor supporting the $(4,6,12)$-fiber
	\item $K_X = h^*(K_W)$
	\item $(g \circ f)(S) = 0 \in T$
	\item $Y$ a terminal model of $X$.
	\item $\hat{S}$ is the strict transform of $S$ and $\hat{S} \rightarrow S$ is a birational morphism with $\hat{S}$ having at worst canonical singularities.
\end{itemize}

\begin{definition}[{cf. \cite[Thm. 6.23]{KollarMori}}]
A terminal model of $X$ is a birational morphism $\pi: Y \rightarrow X$ such that $Y$ has at worst terminal singularities and $\omega_Y = \pi^*(\omega_X)$
\end{definition}

\begin{remark}
This definition, while different from the literature, applies in our case due to our assumption that $W$ has at worst canonical singularities and consequently so does $X$. 
\end{remark}

From the above discussion, since $K_X = h^*(K_W)$, then $K_X$ is a pullback of a Cartier divisor over $T$. This is important since we can obtain a relative minimal model, $Y$, over $T$ by running the relative minimal model program on any smooth resolution $Z$ over $X$ and we will see that the exceptional rational elliptic surface $S \rightarrow \mathbb{P}^1$ will not be contracted in running this MMP. This is why we can assume that $Y$ above is a relative minimal model over $T$ with at worst terminal singularities and contains a birational elliptic surface $\hat{S} \rightarrow S$.

By using standard theory of MMP in dimension $3$, $W$ always has a terminal model since $W$ has at worst canonical singularities. Furthermore, we are able to run a (log) MMP on $Z \rightarrow X$ and $Z \rightarrow W$ where $Z$ is any resolution with at worst terminal singularities. From this set up, we can proceed with the proof proposition \ref{start}.

\begin{proof}[Proof of Prop. \ref{start}]
From the above, there exists a resolution $Y \rightarrow X \rightarrow W \rightarrow T$ and a surface $\hat{S} \rightarrow \mathbb{P}^1$ as a divisor on $Y$ with discrepancy $0$ over $W$. Furthermore, $S \rightarrow \mathbb{P}^1$ has fibers of type $(n_i, m_i, \Delta_i)$ with $0 < n_i < 4$ and $\sum_i n_i = n$ or $0 < m_i < 6$ and $\sum_i m_i = m$ due to the fact that we are assuming that $S$ does not have an isolated $(n,m,12)$-fiber. Since $\sum_i n_i = n = 4$ or $\sum_i m_i = m = 6$, this is only possible if $S$ is defined by the line bundle $\mathcal{O}_{\mathbb{P}^1}(-1)$. Using the canonical bundle formula for Weierstrass models, shows $\omega_S \cong f_S^*(\mathcal{O}_{\mathbb{P}^1}(-1))$ and $S$ is a rational elliptic surface.
\end{proof}

Lastly, the resolution of the isolated $(4,6,12)$-fiber via blowing up the base can be applied to fibers of type $(n,m,d)$ with $4 \leq n < 8$ and $6 \leq m < 12$ which will produce a crepant birational morphism $h:X \rightarrow W$ but $S \rightarrow \mathbb{P}^1$ in this case has generic fiber which is singular. This is useful in constructing examples of Weierstrass threefolds with at worst canonical singularities with the following.

\begin{proposition}
Let $W \rightarrow T$ be a minimal Weierstrass model with discriminant locus having normal crossing with isolated fibers of type $(n,m,d)$ with $n < 8$ and $m < 12$. Then $W$ has at worst canonical singularities.
\end{proposition}

\begin{proof}
Using the above resolution with the arguments in \cite{Miranda83}, we are able to blow up the base until the discriminant locus is simple normal crossing and pull the local equations of $W \rightarrow T$ and reduce orders to obtain a birational morphism $f: X \rightarrow W$ with $K_X = f^*(K_W) + E$ where $E \geq 0$ supports divisors that arise from resolving isolated $(n,m,d)$-fibers where $d < 12$. We can take $X$ to be a Miranda model which is a minimal Weierstrass model with simple normal crossing discriminant locus and having only collision which admit equidimensional resolution as in \cite{Miranda83}. Then from the canonical bundle formula of \cite{Fujita86} there is a smooth resolution $\phi: Y \rightarrow X$ and $K_Y = \phi^*(K_X)$, which shows that exceptional divisors over $W$ is non-negative and so $W$ has at worst canonical singularities.
\end{proof}

\subsection{Flops and $(4,6,12)$-fibers}
With the terminal model, $Y$, above, the natural consideration is the existences of other terminal models, $Z$, of $W$. We can achieve this by flops:
\begin{theorem}[cf. {\cite{Kawamata08}}]
Let $f: Y \rightarrow T$ and $g:Z \rightarrow T$ be projective morphisms with $Y,Z$ having at worst terminal singularities and $K_Y$ and $K_Z$ being relatively nef over $T$. If there is a birational map $\alpha: Y \dashrightarrow Z$ that commutes with $f$ and $g$ then $\alpha$ decomposes into a sequence of flops over $T$, which we can take to be $D$-flops.
\end{theorem}
By a $D$-flop we mean in the sense of \cite[Def. 6.10]{KollarMori}. From the above theorem, we can consider the reverse, where we start with a divisor $D$ and use $D$ to flops curves on $Y$. Now, if we let $D = \hat{S} \subset Y$ then, since $Y$ is relatively minimal over $T$ and $\hat{S} \mapsto 0 \in T$, all curves $C \subset S$ such that $C \cdot S < 0$ are candidates to be flopped on $Y$.

Rational elliptic surfaces with section and their Mordell-Weil groups are well understood, for example in\cite{SchuttShioda19}. In particular, every section of a rational elliptic surface is a $(-1)$-curve and so is contractible. As $Y$ contains $\hat{S}$, this gives candidates for flopping curves on $Y$ as well as relations between the birational geometry of $Y$ and the Mordell-Weil group of the embedded rational elliptic surface. We capture this with the following theorem.

\begin{theorem}
\label{1main}
Using the commutative diagram above, given an isolated $(4,6,12)$-fiber and $\hat{S}$ the resulting rational elliptic surface associated with this fiber and let $C \subset \hat{S}$ be a curve which avoids the singular locus of $Y$. If $C$ is a $K_Y + \epsilon \hat{S}$-negative rational curve that generates the extremal ray $R = [C] \subset \overline{NE}(Y/S)$ then $C$ must be a section of $\hat{S} \rightarrow \mathbb{P}^1$ or a component of the vertical fiber.
\end{theorem}

\begin{proof}
Assume that $C$ is horizontal in $\hat{S} \rightarrow \mathbb{P}^1$ and let $\phi: (Z,S_Z) \rightarrow (Y,\hat{S})$ be a log resolution with $\phi_S: S_Z \rightarrow \hat{S}$ the resolution restricted to $S_Z$ with $C_Z = \phi^{-1}(C)$ in $S_Z$. The projection formula gives the following calculations:
\[
(K_Y + \hat{S}) \cdot C = \phi^*(K_Y + \hat{S}) \cdot C_Z = (K_Z + S_Z - \Gamma) \cdot C_Z
\]
where $\Gamma$ is $\phi$-exceptional. As $C$ avoids the singular locus of $Y$ and $S$, $\Gamma_Y \cdot C_Z \geq 0$ and so with the adjunction formula gives:
\[
(K_Z + S_Z + \Gamma) \cdot C_Z = (K_Z + S_Z) \cdot C_Z = (K_Z + S_Z)|_{S_Z} \cdot C_Z = K_{S_Z} \cdot C_Z
\]
Thus the above calculations show that $K_{S_Z} \cdot C_Z < 0$ as $C$ being horizontal means that it is not contained in the singular locus of $Y$ or $S$. Furthermore, since $C$ was contractible this implies that components of $\phi_S^*(C)$ form a negative definite intersection matrix and thus $C_Z \cdot C_Z < 0$. This is only possible if $C_Z \cong \mathbb{P}^1$ and $C_Z \cdot K_{S_Z} = C_Z \cdot C_Z = -1$. As $C$ and thus $C_Z$ is a horizontal curve this is only possible if they were rational sections of $S \rightarrow \mathbb{P}^1$. 

To show this we can assume $\psi :S_Z \rightarrow \mathbb{P}^1$ is relatively minimal by contracting the $(-1)$-curves in the fibers, which is disjoint from $C_Z$ as $C$ avoids the singular locus of $S$ and $Y$. Then the restriction of $\psi$ onto $C_Z$ gives a surjective map to $\mathbb{P}^1$ which is also injective since $- K_{S_Z} \cdot C_Z = F \cdot C_Z = 1$ where $F$ is a fiber of $\psi$. Thus $\psi|_{C_Z}$ is a bijective morphism onto a smooth curve and so is an isomorphism and it's inverse gives the section map.
\end{proof}

The above has strong conditions on $C$,$S$ and $Y$ for find a flopping curve on $S$. It is known that all the sections in the smooth locus are such candidates and by \cite[Prop.1.29]{CoxZucker} there is always a finite index subgroup of the Mordell-Weil group of $\hat{S}$ which will satisfy this condition, but even so it is not clear if it is an extremal ray in $\overline{NE}(Y/S)$. One benefit of the Weierstrass equation is that there is always a flopping curve.

\begin{proposition}
There is at least $2$ terminal models associated to an isolated $(4,6,12)$-fiber in any Weierstrass threefold $W \rightarrow T$ with at worst canonical singularities.
\end{proposition}

\begin{proof}
Let $\Sigma_0$ be the zero section of $W \rightarrow T$ and $\Sigma := h^*(\Sigma)$ the $0$-section on $X$. As $T$ was smooth, $\Sigma$ is not contained in the singular locus of $X$ or $S$ and so $\sigma = \Sigma \cap S$ is a section of $S$ and is in the smooth locus of $Y$ and $S$. Since $\Sigma \cong \hat{T}$ and $Y$ is smooth along $\Sigma$, running a relative MMP on $(Y, \epsilon \Sigma)$ would correspond to running an MMP from $\hat{T} \rightarrow T$ which would contract $\sigma$. As $\sigma$ generates an extremal ray in $\overline{NE}(Y/T)$, the resulting flip is unique and is contracted irrespective of what divisor is used to contract it, \cite[Cor. 6.4]{KollarMori}, thus it would also be an extremal contraction when running the relative MMP on $(Y, \epsilon S)$.
\end{proof}

More can be said in the above since every elliptic surface appearing from the resolution coming from proposition \ref{start} would contribute a flopping curve. In example \ref{example1}, which requires resolving an isolated $(4,6,12)$-fiber twice, produces two elliptic surfaces, the first which is a product of an elliptic curve and $\mathbb{P}^1$ and the second is a rational elliptic surfaces. Running the relative MMP with $\Sigma$ would contract the $0$-sections of both surfaces after two contractions and so in example \ref{example1} there are at least 3 terminal models.

\subsection{Flops and Generic $(4,6,12)$-fibers}

From above, it can be seen that the use of MMP to obtain different minimal models is constrained by the singularities on $S$ and the singularities of $Y$ along $S$. The generic case is much more well behaved, e.g. if the $(4,6,12)$-fiber was formed by a normal crossing of divisors on $T$. For this assume that $Y$ and $S$ are smooth with $Y$ being smooth along $S$. This will allow for an explicit description of running the relative minimal model program of $(Y,\epsilon S)$ over $T$ by reducing to the case of running an MMP on the surface $S$ itself. 

\begin{theorem}
A $K_Y + S$-MMP over $T$ agrees with a $K_S$-MMP on $S$.
\end{theorem}

\begin{proof}
As $S$ is smooth then $(Y, S)$ is plt around $S$ and we can run an MMP with scaling. Let $A$ be an ample Cartier divisor such that $K_Y + S + A$ is big and nef and let $0 < \lambda < 1$ be such that $(K_Y + S + \lambda A)\cdot C = 0$ where $C$ is a rational curve which generates the extremal ray in $\overline{NE}(Y/T)$. As $K_Y$ is numerically trivial over $T$, this is only possible if $S \cdot C < 0$. 

We relate the $K_Y + S$-MMP over $T$ with a $K_S$-MMP by showing the map
\[
H^0(m(K_Y + S + \lambda A)) \rightarrow H_0(m(K_S + \lambda A|_S))
\]
induced by the exact sequence
\[
0 \rightarrow \mathcal{O}_Y(m(K_Y + S + \lambda A) - S) \rightarrow \mathcal{O}_Y(m(K_Y + S + \lambda A)) \rightarrow \mathcal{O}_S(m(K_S + \lambda A|_S)) \rightarrow 0
\]
is surjective. As $\lambda A + (m-1)(K_Y + S + \lambda A)$ is nef and big, by the Kawamata Viehweg vanishing,  $H^1(K_Y + \lambda A + (m - 1)(K_Y + S + \lambda A)) = 0$ and so $H^0(K_Y + S + \lambda A) \rightarrow H_0(K_S + \lambda A|_S)$ is a surjection. Thus the contraction of $R$ corresponds to a contraction of an extremal ray in $\overline{NE}(S)$ in the $K_S$ negative part. This is a contraction of a curve on $S$ and so is a flopping contraction on $Y$. From before we know these are $(-1)$-curves that are either sections of $S \rightarrow \mathbb{P}^1$ or exceptional curves that contract to smooth points. 

As extremal contractions are unique, \cite[Cor. 6.4]{KollarMori}, the flop above can be realized as a terminal flop of $(Y ,\epsilon S)$ and since $Y$ was smooth along $S$ then the resulting $Y^+$ is smooth along $S^+$, \cite[Thm 6.15]{KollarMori}. Furthermore, $S^+$ was obtained from contracting a $(-1)$-curve and so is smooth, thus we can repeat with $(Y^+, S^+)$ until this process stops when the strict transform of $S$ is contracted. 
\end{proof}

Thus for the generic case, the relative MMP will eventually contract the rational elliptic surface $S$ according to the classical minimal model theory on surfaces. As a consequence, $S$ would eventually be contracted to $\mathbb{P}^2$ or to $\mathbb{P}^1 \times \mathbb{P}^1$ which corresponds to the respective blowing up at $9$ and $8$ points to obtain a minimal rational elliptic surface. From this we have at least a minimum bound on the number of models obtained.

\begin{corollary}
There is at least $9$ terminal models associated to a generic isolated $(4,6,12)$-fiber.
\end{corollary}

\section{Number of Minimal Models and Mordell-Weil Groups}

In this section we discuss the number of minimal models as constrained by the arithmetic of $S$ and of $Y$. In particular, for generic $S$ the arithmetic of $S$ imposes significant geometric constrains on it's sections which dictates how MMP runs on $S$. In this section, we assume that $S$ and $Y$ are smooth, $Y$ is smooth along $S$ and $S \rightarrow \mathbb{P}^1$ is relatively minimal. The last assumption is not too strong of an assumption since using the above MMP we can flop the $(-1)$-curves on $S$ contained within the fiber leaving us in this current situation where $S$ is relatively minimal. 

\begin{definition}
Let $f: Y \rightarrow T$ be an elliptic fibration with section and $Y_\eta$, the generic fiber, then the Mordell-Weil group, $MW(X_\eta) = MS(X/T)$, is the group of $\eta$-points on $X_\eta$ which corresponds to sections of a birationally equivalent Weierstrass model $W \rightarrow T$.
\end{definition}

The Mordell-Weil group of an elliptic fibration is a birational invaraint over birationall equivalent fibrations. It carries more information that just points since we can realize $MW(Y/T)$ as a finite index subgroup of $Bir(Y/T)$ the relative birational automorphism group of $Y$ over $T$.

\begin{proposition}
Let $Y \rightarrow T$ be an elliptic fibration with section and $\eta \in T$ the generic point with $Y_\eta$ the generic fiber, which is an elliptic curve over $\eta$, then:
\[
Aut(Y_\eta) = Bir(Y/B)
\]
\end{proposition}

\begin{proof}
A relative birational automorphism, that is an isomorphism in codimension 1, would preserve the generic fiber, thus would restrict to an automorphism on the generic fiber. Furthermore, an automorphism of the generic fiber gives an action on the smooth fibers of $Y \rightarrow T$ which produces a birational automorphism over $T$. 
\end{proof}

The existence of the $(4,6,12)$-singular fiber on $W \rightarrow T$ gives two Mordell-Weil grouips to study, specifically $MW(Y/T) \cong MW(W/T)$ and $MW(S/\mathbb{P}^1)$. The action of $MW(Y/T)$ acts on the smooth fibers of $Y \rightarrow T$ and since $Y \rightarrow \hat{T}$ is a birationally equivalent model the action extends to the smooth fibers of $S \rightarrow \mathbb{P}^1$. This gives a group morphism $MW(Y/T) \rightarrow MW(S/\mathbb{P}^1)$. This sets us up for the main theorem.

\begin{theorem}
\label{2main}
If there is an infinite number of flopping curves on $S$ then $MW(Y/T)$ has positive rank and there exists an infinite subgroup of $MW(S/\mathbb{P}^1)$ that consists of flopping section.
\end{theorem}

\begin{proof}
If there is an infinite number of flopping curves on $S$ then since there are only finitely many vertical fibers, this implies that there are infinitely many sections on $S$ which are flopping curves. By \cite{Kawamata97}, there are only finitely many models up to isomorphisms and the number of these isomorphism classes corresponds to the number of $Bir(Y/S)$-orbits of flopping curves. This is only possible if there were an infinite number of relative birational morphisms. As $MW(Y/S)$ is a finite index subgroup of $Bir(Y/S)$, this implies that $MW(Y/T)$ has positive rank.

Now consider the orbit of the $0$-section of $S$ (which we know is a flopping curve) under the action of $MW(Y/T)$. There is a morphism of groups $MW(Y/S) \rightarrow \Lambda \leq MW(S/\mathbb{P}^1)$ obtained by restricting the action of $MW(Y/S)$ on the smooth fibers of $Y \rightarrow T$ to the smooth fibers of $S \rightarrow \mathbb{P}^1$. Then $\Lambda$ must be an infinite subgroup of $MW(S/\mathbb{P}^1)$ which acts by mapping the flopping curves of the $MW(Y/S)$ orbit of the $0$-section to other flopping curve, but then the action of $\Lambda$ on the $0$-section gives exactly the sections corresponding to $\Lambda$ and so there is a subgroup of $MW(S/\mathbb{P}^1)$ consisting of flopping sections on $Y$.
\end{proof}

\begin{proposition}
Let $Y \dashrightarrow Y'$ be a flopping contraction of a section on $S$ and let $S'$ be it's strict transform on $Y'$, then the convex hull of the extremal rays in the $K_{Y'} + S'$-negative part of the cone of the relative cone of curves $\overline{NE}(Y'/T)$ is a rational polyhedral cone.
\end{proposition}

\begin{proof}
From the assumptions of $S$ and $Y$ being smooth, after flopping a curve the only extremal rays are curves on $S'$ which are $(-1)$-curves. Given a section of a relatively minimal rational elliptic surface, there are at most 240 sections which are disjoint from a fixed section. These plus the vertical fibers are the only possible flopping curves and so there are finitely many extremal rays in the $K_{Y'} + S'$-negative part of the cone of the relative cone of curves $\overline{NE}(Y'/T)$. thus they're convex hull form a rational polyhedral cone.
\end{proof}

The above shows that the only means of which to obtain and infinite number of models via flops have to come at the stages $(Y,S)$ and not in any further birational models obtained from running the relative MMP. As a consequence we have the following.

\begin{corollary}
If the convex hull of the extremal rays of the $K_Y + S$-negative part of $\overline{NE}(Y/S)$ is rational polyhedral then the number of models arising from flopping curves on $S$ is finite.
\end{corollary}

\begin{corollary}
If $rk(MW(Y/S)) = 0$ then there are only finitely many flopping curves of $Y$ in $S$.
\end{corollary}

This must hold even if the number of sections of $S$ is infinite. As such the embedding of $S$ into $Y$ prescribes the distinguished sections which are to be flopped to obtain different minimal models. 

\begin{example}
Consider the example from \cite[Section 11]{HK}. Let $T = Bl_0(\mathbb{P}^2)$ and $\mathcal{L} = \omega_T$, then any Weierstrass model $W := W(\mathcal{L}, f,g)$ with $f$ and $g$ generically chosen in $\mathcal{L}^{-4}$ and $\mathcal{L}^{6}$ has at worst canonical singularities and is birational to a Weierstrass model over $\mathbb{P}^2$ which has an $(4,6,12)$-fiber over $0 \in \mathbb{P}^2$. The corresponding rational elliptic surface is generic since $f$ and $g$ is generic and so has Mordell-Weil group $\mathbb{Z}^8$.

Using the methods of \cite{HK}, it can be shown that the Mordell Weil group of $W \rightarrow T$ is trivial. Even though the rational elliptic surface $S$ over $0$ has Mordell-Weil rank $8$ and so there is an infinite number of sections yet only finitely many are extremal rays in the relative cone of curves.
\end{example}

If the $(4,6,12)$-fiber was formed by normal crossings and produces the extremal rational elliptic surfaces which are rational elliptic surfaces with fintie Mordell-Weil group. Then the number of rantional curves with negative self intersection is finite, \cite{MR}, and they consist of vertical fibers and the torsion sections (which are disjoint from each other). Thus in these cases the number of models arising from flops on $S$ is finite.

\begin{example}
Let $T = \mathbb{P}^2$, $\mathcal{L} = \omega_T$ and $\alpha, \beta \in H^0(\mathcal{L}^{-1})$ be general global sections, then we can define the Weierstrass model $W := W(\mathcal{L}, \alpha^2\beta^2, \alpha^3\beta^3) \rightarrow T$. This is an elliptic fibrations such that over a general fiber of $div(\alpha) \cup div(\beta) \subset T$, is a $(2,3,6)$-singular fibers which resolves to a $I_0^*$ singular fiber. Over the $9$ collision points of $div(\alpha) \cap div(\beta)$ are $(4,6,12)$-singular fibers. 

The singular locus of $W$ is exactly the singular points of the singular fibers of $W \rightarrow T$. We can resolve $W$ as follow, first we can apply the weighted blow up to the $9$ $(4,6,12)$-singular fibers resulting in $9$ rational elliptic surfaces. Then the singular locus after the weight blow up becomes a disjoint smooth curves isomorphic to $div(\alpha)$ and $div(\beta)$ respectively, which we can resolve similar to the elliptic surface case to get $I_0^*$ singular fibers generically over $div(\alpha) \cup div(\beta)$. This gives a smooth crepant resolution $X \rightarrow W \rightarrow T$, and on $X$ there exists $9$ smooth rational elliptic surfaces that map to points on $T$ and each with only $2$ $I_0^*$ singular fibers.

From the classification in \cite{OguisoShioda}, these rational elliptic surface admit only $4$ sections with the singular fibers being of type  $D_4$ having $5$ components. Thus there is a total of $14$ rational curves on $S$, which implies that there are only finitely many possible contractions of rational curves. This gives an upper bound of ${ 14 \choose 9 }2^9$ possible combinations of contractions on the rational elliptic surface. As each surface is independent of each other, the number of relative minimal models of $W_T \rightarrow \mathbb{P}^2$ is bounded above by $9 \cdot { 14 \choose 9 }2^9$. Furthermore, as each rational elliptic surface is smooth, there is at least $9^9$ different minimal models since each rational elliptic surface permits up to $9$ flops.
\end{example}

\bibliographystyle{plain}
\bibliography{refs}{}

\end{document}